\documentclass[11pt]{article}
\usepackage{amsmath,amsthm,amssymb}
\usepackage[dvips]{graphicx}

\newcommand{\R}{\mathbb R}
\newcommand{\C}{\mathbb C}
\newcommand{\HH}{{\cal H}}
\newcommand{\T}{{\cal T}}
\newcommand{\U}{{\cal U}}
\newcommand{\PP}{{\mathbb P}}
\newcommand{\re}{{\operatorname{Re}}}
\newcommand{\im}{{\operatorname{Im}}}
\newcommand{\ii}{{\operatorname{i}}}

\begin{document}


\title{Geometric Properties of Inverse Polynomial Images\footnote{published in: Proceedings Approximation Theory XIII: San Antonio 2010, {\it Springer Proceedings in Mathematics} {\bf 13} (2012), 277--287.}}
\author{Klaus Schiefermayr\footnote{University of Applied Sciences Upper Austria, School of Engineering and Environmental Sciences, Stelzhamerstrasse\,23, 4600 Wels, Austria, \textsc{klaus.schiefermayr@fh-wels.at}}}
\date{}
\maketitle

\theoremstyle{plain}
\newtheorem{theorem}{Theorem}
\newtheorem{corollary}{Corollary}
\newtheorem{lemma}{Lemma}
\newtheorem{definition}{Definition}
\theoremstyle{definition}
\newtheorem{remark}{Remark}
\newtheorem{example}{Example}

\begin{abstract}
Given a polynomial $\T_n$ of degree $n$, consider the inverse image of $\R$ and $[-1,1]$, denoted by $\T_n^{-1}(\R)$ and $\T_n^{-1}([-1,1])$, respectively. It is well known that $\T_n^{-1}(\R)$ consists of $n$ analytic Jordan arcs moving from $\infty$ to $\infty$. In this paper, we give a necessary and sufficient condition such that (1)~$\T_n^{-1}([-1,1])$ consists of $\nu$ analytic Jordan arcs and (2)~$\T_n^{-1}([-1,1])$ is connected, respectively.
\end{abstract}

\noindent\emph{Mathematics Subject Classification (2000):} 30C10, 30E10

\noindent\emph{Keywords:} Analytic Jordan arc, Inverse polynomial image

\section{Introduction}


Let $\PP_n$ be the set of all polynomials of degree $n$ with complex coefficients. For a polynomial $\T_n\in\PP_n$, consider the inverse images $\T_n^{-1}(\R)$ and $\T_n^{-1}([-1,1])$, defined by
\begin{equation}
\T_n^{-1}(\R):=\bigl\{z\in\C:\T_n(z)\in\R\bigr\}
\end{equation}
and
\begin{equation}
\T_n^{-1}([-1,1]):=\bigl\{z\in\C:\T_n(z)\in[-1,1]\bigr\},
\end{equation}
respectively. It is well known that $\T_n^{-1}(\R)$ consists of $n$ analytic Jordan arcs moving from $\infty$ to $\infty$ which cross each other at points which are zeros of the derivative $\T_n'$. In \cite{Peh-2003}, Peherstorfer proved that $\T_n^{-1}(\R)$ may be split up into $n$ Jordan arcs (not necessarily analytic) moving from $\infty$ to $\infty$ with the additional property that $\T_n$ is strictly monotone decreasing from $+\infty$ to $-\infty$ on each of the $n$ Jordan arcs. Thus, $\T_n^{-1}([-1,1])$ is the union of $n$ (analytic) Jordan arcs and is obtained from $\T_n^{-1}(\R)$ by cutting off the $n$ arcs of $\T_n^{-1}(\R)$. In \cite[Thm.\,3]{PehSch-2004}, we gave a necessary and sufficient condition such that $\T_n^{-1}([-1,1])$ consists of $2$ Jordan arcs, compare also \cite{Pakovich-1995}, where the proof can easily be extended to the case of $\ell$ arcs, see also \cite[Remark after Corollary\,2.2]{Peh-2003}. In the present paper, we will give a necessary and sufficient condition such that (1)~$\T_n^{-1}([-1,1])$ consists of $\nu$ (but not less than $\nu$) \emph{analytic} Jordan arcs (in Section\,2) and (2)~$\T_n^{-1}([-1,1])$ is connected (in Section\,3), respectively. From a different point of view as in this paper, inverse polynomial images are considered, e.g., in \cite{PehSt}, \cite{Pakovich-1998}, \cite{Pakovich-2007}, and \cite{Pakovich-2008}.

Inverse polynomial images are interesting for instance in approximation theory, since each polynomial (suitable normed) of degree $n$ is the minimal polynomial with respect to the maximum norm on its inverse image, see \cite{KamoBorodin}, \cite{Peh-1996}, \cite{FischerPeherstorfer}, and \cite{Fischer-1992}.

\section{The Number of (Analytic) Jordan Arcs of an Inverse Polynomial Image}


Let us start with a collection of important properties of the inverse images $\T_n^{-1}(\R)$ and $\T_n^{-1}([-1,1])$. Most of them are due to Peherstorfer\,\cite{Peh-2003} or classical well known results. Let us point out that $\T_n^{-1}(\R)$ (and also $\T_n^{-1}([-1,1])$), on the one hand side, may be characterized by $n$ analytic Jordan arcs and, on the other side, by $n$ (not necessarily analytic) Jordan arcs, on which $\T_n$ is strictly monotone.

Let $C:=\{\gamma(t):t\in[0,1]\}$ be an analytic Jordan arc in $\C$ and let $\T_n\in\PP_n$ be a polynomial such that $\T_n(\gamma(t))\in\R$ for all $t\in[0,1]$. We call a point $z_0=\gamma(t_0)$ a \emph{saddle point} of $\T_n$ on $C$ if $\T_n'(z_0)=0$ and $z_0$ is no extremum of $\T_n$ on $C$.


\begin{lemma}\label{Lem-1}
Let $\T_n\in\PP_n$ be a polynomial of degree $n$.
\begin{enumerate}
\item $\T_n^{-1}(\R)$ consists of $n$ analytic Jordan arcs, denoted by $\tilde{C}_1,\tilde{C}_2,\dots,\tilde{C}_n$, in the complex plane running from $\infty$ to $\infty$.
\item $\T_n^{-1}(\R)$ consists of $n$ Jordan arcs, denoted by $\tilde{\Gamma}_1,\tilde{\Gamma}_2,\dots,\tilde{\Gamma}_n$, in the complex plane running from $\infty$ to $\infty$, where on each $\tilde{\Gamma}_j$, $j=1,2,\ldots,n$, $\T_n(z)$ is strictly monotone decreasing from $+\infty$ to $-\infty$.
\item A point $z_0\in\T_n^{-1}(\R)$ is a crossing point of exactly $m$, $m\geq2$, analytic Jordan arcs $\tilde{C}_{i_1},\tilde{C}_{i_2},\ldots,\tilde{C}_{i_m}$, $1\leq{i}_1<i_2<\ldots<i_m\leq{n}$, if and only if $z_0$ is a zero of $\T'_n$ with multiplicity $m-1$. In this case, the $m$ arcs are cutting each other at $z_0$ in successive angles of $\pi/m$. If $m$ is odd then $z_0$ is a saddle point of $\re\{\T_n(z)\}$ on each of the $m$ arcs. If $m$ is even then, on $m/2$ arcs, $z_0$ is a minimum of $\re\{\T_n(z)\}$ and on the other $m/2$ arcs, $z_0$ is a maximum of $\re\{\T_n(z)\}$.
\item A point $z_0\in\T_n^{-1}(\R)$ is a crossing point of exactly $m$, $m\geq2$, Jordan arcs\\ $\tilde{\Gamma}_{i_1},\tilde{\Gamma}_{i_2},\ldots,\tilde{\Gamma}_{i_m}$, $1\leq{i}_1<i_2<\ldots<i_m\leq{n}$, if and only if $z_0$ is a zero of $\T'_n$ with multiplicity $m-1$.
\item $\T_n^{-1}([-1,1])$ consists of $n$ analytic Jordan arcs, denoted by $C_1,C_2,\dots,C_n$, where the $2n$ zeros of $\T_n^2-1$ are the endpoints of the $n$ arcs. If $z_0\in\C$ is a zero of $\T_n^2-1$ of multiplicity $m$ then exactly $m$ analytic Jordan arcs $C_{i_1},C_{i_2},\ldots,C_{i_m}$ of $\T_n^{-1}([-1,1])$, $1\leq{i}_1<i_2<\ldots<i_m\leq{n}$, have $z_0$ as common endpoint.
\item $\T_n^{-1}([-1,1])$ consists of $n$ Jordan arcs, denoted by $\Gamma_1,\Gamma_2,\dots,\Gamma_n$, with $\Gamma_j\subset\tilde{\Gamma}_j$, $j=1,2,\dots,n$, where on each $\Gamma_j$, $\T_n(z)$ is strictly monotone decreasing from $+1$ to $-1$. If $z_0\in\C$ is a zero of $\T_n^2-1$ of multiplicity $m$, then exactly $m$ Jordan arcs $\Gamma_{i_1},\ldots,\Gamma_{i_m}$ of $\T_n^{-1}([-1,1])$, $1\leq{i}_1<i_2<\ldots<i_m\leq{n}$, have $z_0$ as common endpoint.
\item Two arcs $C_j,C_k$, $j\neq k$, cross each other at most once (the same holds for $\Gamma_j,\Gamma_k$).
\item Let $S:=\T_n^{-1}([-1,1])$, then the complement $\C\setminus{S}$ is connected.
\item Let $S:=\T_n^{-1}([-1,1])$, then, for $P_n(z):=\T_n((z-b)/a)$, $a,b\in\C$, $a\neq0$, the inverse image is $P_n^{-1}([-1,1])=aS+b$.
\item $\T_n^{-1}([-1,1])\subseteq\R$ if and only if the coefficients of $\T_n$ are real, $\T_n$ has $n$ simple real zeros and $\min\bigl\{|\T_n(z)|:\T'_n(z)=0\bigr\}\geq1$.
\item $\T_n^{-1}(\R)$ is symmetric with respect to the real line if and only if $\T_{n}(z)$ or $\ii\T_{n}(z)$ has real coefficients only.
\end{enumerate}
\end{lemma}
\begin{proof}
(i), (iii), (iv), and (xi) are well known.\\
For (ii), see \cite[Thm.\,2.2]{Peh-2003}.\\
Concerning the connection between (iii),(iv) and (v),(vi) note that each zero $z_0$ of $Q_{2n}(z)=\T_n^2(z)-1\in\PP_{2n}$ with multiplicity $m$ is a zero of $Q_{2n}'(z)=2\T_n(z)\,\T_n'(z)$ with multiplicity $m-1$, hence a zero of $\T_n'(z)$ with multiplicity $m-1$. Thus, (v) and (vi) follow immediately from (i)\&(iii) and (ii)\&(iv), respectively.\\
(vii) follows immediately from (viii).\\
Concerning (viii), suppose that there exists a simple connected domain $B$, which is surrounded by a subset of $\T_n^{-1}([-1,1])$. Then the harmonic function $v(x,y):=\im\{\T_n(x+\ii{y})\}$ is zero on $\partial{B}$ thus, by the maximum principle, $v(x,y)$ is zero on $B$, which is a contradiction.\\
(ix) follows from the definition of $\T_n^{-1}([-1,1])$.\\
For (x), see \cite[Cor.\,2.3]{Peh-2003}.
\end{proof}


\begin{example}\label{Ex}
Consider the polynomial $\T_n(z):=1+z^2(z-1)^3(z-2)^4$ of degree $n=9$. Fig.\,\ref{Fig_InverseImageGeneral} shows the inverse images $\T_n^{-1}([-1,1])$ (solid line) and $\T_n^{-1}(\R)$ (dotted and solid line). The zeros of $\T_n+1$ and $\T_n-1$ are marked with a circle and a disk, respectively. One can easily identitfy the $n=9$ analytic Jordan arcs $\tilde{C}_1,\tilde{C}_2,\ldots,\tilde{C}_n$ which $\T_n^{-1}(\R)$ consists of, compare Lemma\,\ref{Lem-1}\,(i), and the $n=9$ analytic Jordan arcs $C_1,C_2,\ldots,C_n$ which $\T_n^{-1}([-1,1])$ consists of, compare Lemma\,\ref{Lem-1}\,(v), where the endpoints of the arcs are exactly the circles and disks, i.e., the zeros of $\T_n^2-1$. Note that $\tilde{C}_1=\R$, $C_1=[-0.215\ldots,0]$ and $C_2=[0,1]$.
\end{example}


\begin{figure}[ht]
\begin{center}
\includegraphics[scale=1.6]{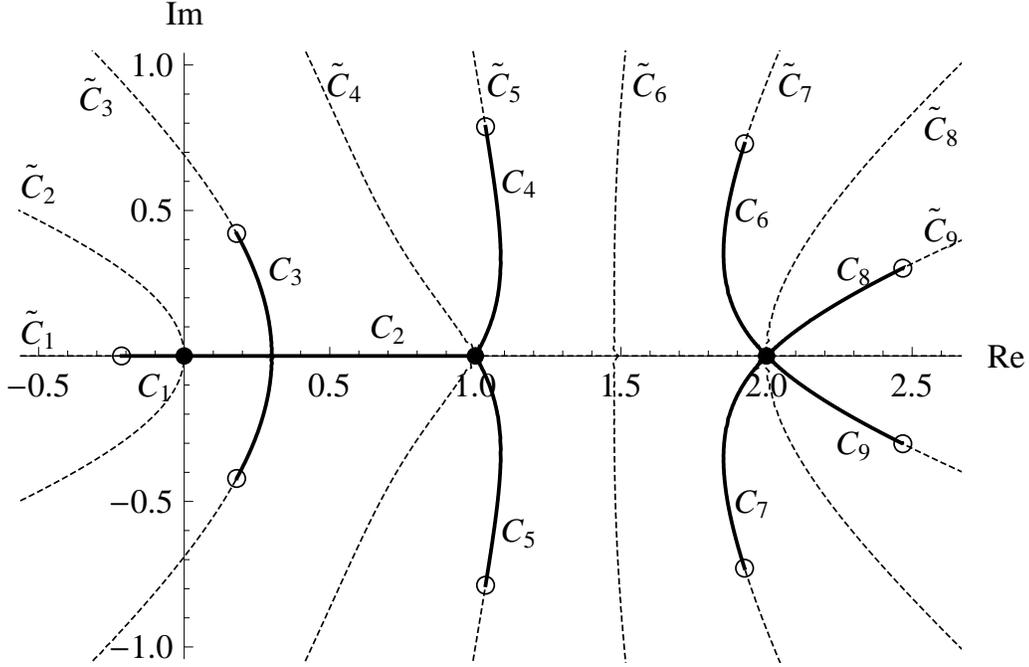}
\caption{\label{Fig_InverseImageGeneral} Inverse images $\T_9^{-1}([-1,1])$ (solid line) and $\T_9^{-1}(\R)$ (dotted and solid line) for the polynomial $\T_9(z):=1+z^2(z-1)^3(z-2)^4$}
\end{center}
\end{figure}


Before we state the result concerning the minimal number of analytic Jordan arcs $\T_n^{-1}([-1,1])$ consists of, let us do some preparations. Let $\T_n\in\PP_n$ and consider the zeros of the polynomial $\T_n^2-1\in\PP_{2n}$. Let $\{a_1,a_2,\ldots,a_{2\ell}\}$ be the set of all zeros of $\T_n^2-1$ with \emph{odd} multiplicity, where $a_1,a_2,\ldots,a_{2\ell}$ are pairwise distinct and each $a_j$ has multiplicity $2\beta_j-1$, $j=1,\ldots,2\ell$. Further, let
\begin{equation}
(b_1,b_2,\ldots,b_{2\nu}):=(\underbrace{a_1,\ldots,a_1}_{(2\beta_1-1)-\text{times}},
\underbrace{a_2,\ldots,a_2}_{(2\beta_2-1)-\text{times}},\ldots,\underbrace{a_{2\ell},\ldots,a_{2\ell}}_{(2\beta_{2\ell}-1)-\text{times}}),
\end{equation}
thus
\begin{equation}
2\nu=\sum_{j=1}^{2\ell}(2\beta_j-1),
\end{equation}
i.e., $b_1,b_2,\ldots,b_{2\nu}$ are the zeros of odd multiplicity \emph{written according to their multiplicity}.


\begin{theorem}\label{Thm-NuArcs}
Let $\T_n\in\PP_n$ be any polynomial of degree $n$. Then, $\T_n^{-1}([-1,1])$ consists of $\nu$ (but not less than $\nu$) analytic Jordan arcs with endpoints $b_1,b_2,\ldots,b_{2\nu}$ if and only if $\T_n^2-1$ has exactly $2\nu$ zeros $b_1,b_2,\ldots,b_{2\nu}$ (written according to their multiplicity) of odd multiplicity.
\end{theorem}
\begin{proof}
By Lemma\,\ref{Lem-1}\,(v), $\T_n^{-1}([-1,1])$ consists of $n$ analytic Jordan arcs $C_1,C_2,\ldots,C_n$, which can be combined into $\nu$ analytic Jordan arcs in the following way. Clearly, two analytic Jordan arcs $C_{i_1}$ and $C_{i_2}$ can be joined together into one analytic Jordan arc if they have the same endpoint, which is a zero of $\T_n^2-1$, and if they lie on the same analytic Jordan arc $\tilde{C}_{i_3}$ of Lemma\,\ref{Lem-1}\,(i). By Lemma\,\ref{Lem-1}\,(iii) and (v), such combinations are possible only at the zeros of $\T_n^2-1$ of \emph{even} multiplicity. More precisely, let $d_1,d_2,\ldots,d_k$ be the zeros of $\T_n^2-1$ with even multiplicities $2\alpha_1,2\alpha_2,\ldots,2\alpha_k$, where, by assumption,
\[
2\alpha_1+2\alpha_2+\ldots+2\alpha_k=2n-2\nu.
\]
By Lemma\,\ref{Lem-1}\,(iii) and (v), at each point $d_j$, the $2\alpha_j$ analytic Jordan arcs of $\T_n^{-1}([-1,1])$ can be combined into $\alpha_j$ analytic arcs, $j=1,2,\ldots,k$. Altogether, the number of such combinations is $\alpha_1+\alpha_2+\ldots+\alpha_k=n-\nu$, thus the the total number of $n$ analytic Jordan arcs is reduced by $n-\nu$, hence $\nu$ analytic Jordan arcs remain and the sufficiency part is proved. Since, for each polynomial $\T_n\in\PP_n$, there is a unique $\nu\in\{1,2,\ldots,n\}$ such that $\T_n^2-1$ has exactly $2\nu$ zeros of odd multiplicity (counted with multiplicity), the necessity part follows.
\end{proof}


\begin{example}
For a better understanding of the combination of two analytic Jordan arcs into one analytic Jordan arc, as done in the proof of Theorem\,\ref{Thm-NuArcs}, let us again consider the inverse image of the polynomial of Example\,\ref{Ex}.
\begin{itemize}
\item The point $d_1=0$ is a zero of $\T_n-1$ with multiplicity $2\alpha_1=2$, thus $2$ analytic Jordan arcs, here $C_1$ and $C_2$, have $d_1$ as endpoint, compare Lemma\,\ref{Lem-1}\,(v). Along the arc $\tilde{C}_1$, $d_1$ is a maximum, along the arc $\tilde{C}_2$, $d_1$ is a minimum, compare Lemma\,\ref{Lem-1}\,(iii), thus the $2$ analytic Jordan arcs $C_1$ and $C_2$ can be joined together into one analytic Jordan arc $C_1\cup{C}_2$.
\item The point $d_2=2$ is a zero of $\T_n-1$ with multiplicity $2\alpha_2=4$, thus $4$ analytic Jordan arcs, here $C_6$, $C_7$, $C_8$ and $C_9$, have $d_2$ as endpoint. Along the arc $\tilde{C}_7$ or $\tilde{C}_9$, $d_3$ is a maximum, along the arc $\tilde{C}_8$ or $\tilde{C}_1$, $d_3$ is a minimum, compare Lemma\,\ref{Lem-1}\,(iii). Hence, the analytic Jordan arcs $C_6$ and $C_9$ can be combined into one analytic Jordan arc $C_6\cup{C}_9$, analogously $C_7$ and $C_8$ can be combined into $C_7\cup{C}_8$.
\item The point $a_1=1$ is a zero of $\T_n-1$ with multiplicity $3$, thus $3$ analytic Jordan arcs, here $C_2$, $C_4$ and $C_5$, have $a_1$ as endpoint. Since $a_1$ is a saddle point along each of the three analytic Jordan arcs $\tilde{C}_1,\tilde{C}_4,\tilde{C}_5$, compare Lemma\,\ref{Lem-1}\,(iii), no combination of arcs can be done.
\end{itemize}
Altogether, we get $\alpha_1+\alpha_2=3=n-\nu$ combinations and therefore $\T_n^{-1}([-1,1])$ consists of $\nu=6$ analytic Jordan arcs, which are given by $C_1\cup{C}_2$, $C_3$, $C_4$, $C_5$, $C_6\cup{C}_9$ and $C_7\cup{C}_8$.
\end{example}


\begin{lemma}\label{Lemma-PolEq}
For any polynomial $\T_n(z)=c_nz^n+\ldots\in\PP_n$, $c_n\in\C\setminus\{0\}$, there exists a unique $\ell\in\{1,2,\ldots,n\}$, a unique monic polynomial $\HH_{2\ell}(z)=z^{2\ell}+\ldots\in\PP_{2\ell}$ with pairwise distinct zeros $a_1,a_2,\ldots,a_{2\ell}$, i.e.,
\begin{equation}\label{H}
\HH_{2\ell}(z)=\prod_{j=1}^{2\ell}(z-a_j),
\end{equation}
and a unique polynomial $\U_{n-\ell}(z)=c_nz^{n-\ell}+\ldots\in\PP_{n-\ell}$ with the same leading coefficient $c_n$ such that the polynomial equation
\begin{equation}\label{TU}
\T_n^2(z)-1=\HH_{2\ell}(z)\,\U_{n-\ell}^2(z)
\end{equation}
holds. Note that the points $a_1,a_2,\ldots,a_{2\ell}$ are exactly those zeros of $\T_n^2-1$ which have odd multiplicity.
\end{lemma}
\begin{proof}
The assertion follows immediately by the fundamental theorem of algebra for the polynomial $Q_{2n}(z):=\T_n^2(z)-1=c_n^2z^{2n}+\ldots\in\PP_{2n}$, where $2\ell$ is the number of distinct zeros of $Q_{2n}$ with odd multiplicity. It only remains to show that the case $\ell=0$ is not possible. If $\ell=0$, then all zeros of $Q_{2n}$ are of even multiplicity. Thus there are at least $n$ zeros (counted with multiplicity) of $Q_{2n}'$ which are also zeros of $Q_{2n}$ but not zeros of $\T_n$. Since $Q_{2n}'(z)=2\,\T_n(z)\,\T_n'(z)$, there are at least $n$ zeros (counted with multiplicity) of $\T_n'$, which is a contradiction.
\end{proof}


Let us point out that the polynomial equation \eqref{TU} (sometimes called Pell equation) is the starting point for investigations concerning minimal or orthogonal polynomials on several intervals, see, e.g., \cite{Bogatyrev}, \cite{Peh-1993}, \cite{Peh-1996}, \cite{Peh-2001}, \cite{PehSch-1999}, \cite{SoYu-1995}, and \cite{Totik-2001}.\\
\indent In \cite[Theorem\,3]{PehSch-2004}, we proved that the polynomial equation \eqref{TU} (for $\ell=2$) is equivalent to the fact that $\T_n^{-1}([-1,1])$ consists of $2$ Jordan arcs (not necessarily analytic), compare also \cite{Pakovich-1995}. The condition and the proof can be easily extended to the general case of $\ell$ arcs, compare also \cite[Remark after Corollary\,2.2]{Peh-2003}. In addition, we give an alternative proof similar to that of Theorem\,\ref{Thm-NuArcs}.


\begin{theorem}\label{Thm-EllArcs}
Let $\T_n\in\PP_n$ be any polynomial of degree $n$. Then $\T_n^{-1}([-1,1])$ consists of $\ell$ (but not less than $\ell$) Jordan arcs with endpoints $a_1,a_2,\ldots,a_{2\ell}$ if and only if $\T_n^2-1$ has exactly $2\ell$ pairwise distinct zeros $a_1,a_2,\ldots,a_{2\ell}$, $1\leq\ell\leq{n}$, of odd multiplicity, i.e., if and only if $\T_n$ satisfies a polynomial equation of the form \eqref{TU} with $\HH_{2\ell}$ given in \eqref{H}.
\end{theorem}
\begin{proof}
By Lemma\,\ref{Lem-1}\,(vi), $\T_n^{-1}([-1,1])$ consists of $n$ Jordan arcs $\Gamma_1,\Gamma_2,\dots,\Gamma_n$, which can be combined into $\ell$ Jordan arcs in the following way: Let $d_1,d_2,\ldots,d_k$ be those zeros of $\T_n^2-1$ with \emph{even} multiplicities $2\alpha_1,2\alpha_2,\ldots,2\alpha_k$ and let, as assumed in the Theorem, $a_1,a_2,\ldots,a_{2\ell}$ be those zeros of $\T_n^2-1$ with \emph{odd} multiplicities $2\beta_1-1,2\beta_2-1,\ldots,2\beta_{2\ell}-1$, where
\begin{equation}\label{SumMult}
2\alpha_1+2\alpha_2+\ldots+2\alpha_k+(2\beta_1-1)+(2\beta_2-1)+\ldots+(2\beta_{2\ell}-1)=2n
\end{equation}
holds. By Lemma\,\ref{Lem-1}\,(vi), at each point $d_j$, the $2\alpha_j$ Jordan arcs can be combined into $\alpha_j$ Jordan arcs, $j=1,2,\ldots,\nu$, and at each point $a_j$, the $2\beta_j-1$ Jordan arcs can be combined into $\beta_j$ Jordan arcs, $j=1,2,\ldots,2\ell$. Altogether, the number of such combinations, using \eqref{SumMult}, is
\[
\alpha_1+\alpha_2+\ldots+\alpha_{\nu}+(\beta_1-1)+(\beta_2-1)+\ldots+(\beta_{2\ell}-1)=(n+\ell)-2\ell=n-\ell,
\]
i.e., the total number $n$ of Jordan arcs is reduced by $n-\ell$, thus $\ell$ Jordan arcs remain and the sufficiency part is proved. Since, by Lemma\,\ref{Lemma-PolEq}, for each polynomial $\T_n\in\PP_n$ there is a unique $\ell\in\{1,2,\ldots,n\}$ such that $\T_n^2-1$ has exactly $2\ell$ distinct zeros of odd multiplicity, the necessity part is clear.
\end{proof}


\begin{example}
Similar as after the proof of Theorem\,\ref{Thm-NuArcs}, let us illustrate the combination of Jordan arcs by the polynomial of Example\,\ref{Ex}. Taking a look at Fig.\,\ref{Fig_InverseImageGeneral}, one can easily identitfy the $n=9$ Jordan arcs $\Gamma_1,\Gamma_2,\ldots,\Gamma_n\in\T_n^{-1}([-1,1])$, where each arc $\Gamma_j$ runs from a disk to a circle. Note that the two arcs, which cross at $z\approx0.3$, may be chosen in two different ways. Now, $\T_n^2-1$ has the zero $d_1=0$ with multiplicity $2\alpha_1=2$, the zero $d_2=2$ with multiplicity $2\alpha_2=4$, and a zero $a_1=1$ with multiplicity $2\beta_1-1=3$, all other zeros $a_j$ have multiplicity $2\beta_j-1=1$, $j=2,3,\ldots,2\ell$. Thus, it is possible to have one combination at $d_1=0$, two combinations at $d_2=2$ and one combination of Jordan arcs at $a_1=1$. Altogether, we obtain $\alpha_1+\alpha_2+(\beta_1-1)=4=n-\ell$ combinations and the number of Jordan arcs is $\ell=5$.
\end{example}


For the sake of completeness, let us mention two simple special cases, first the case $\ell=1$, see, e.g., \cite[Remark\,4]{PehSch-2004}, and second, the case when all endpoints $a_1,a_2,\ldots,a_{2\ell}$ of the arcs are real, see \cite{Peh-1993}.


\begin{corollary}
Let $\T_n\in\PP_n$.
\begin{enumerate}
\item $\T_n^{-1}([-1,1])$ consists of $\ell=1$ Jordan arc with endpoints $a_1,a_2\in\C$, $a_1\neq{a}_2$, if and only if $\T_n$ is the classical Chebyshev polynomial of the first kind (suitable normed), i.e., $\T_n(z)=T_n((2z-a_1-a_2)/(a_2-a_1))$, where $T_n(z):=\cos(n\arccos{z})$. In this case, $\T_n^{-1}([-1,1])$ is the complex interval $[a_1,a_2]$.
\item $\T_n^{-1}([-1,1])=[a_1,a_2]\cup[a_3,a_4]\cup\ldots\cup[a_{2\ell-1},a_{2\ell}]$, $a_1,a_2,\ldots,a_{2\ell}\in\R$, $a_1<a_2<\ldots<a_{2\ell}$, if and only if $\T_n$ satisfies the polynomial equation \eqref{TU} with $\HH_{2\ell}$ as in \eqref{H} and $a_1,a_2,\ldots,a_{2\ell}\in\R$, $a_1<a_2<\ldots<a_{2\ell}$.
\end{enumerate}
\end{corollary}


Let us consider the case of $\ell=2$ Jordan arcs in more detail. Given four pairwise distinct points $a_1,a_2,a_3,a_4\in\C$ in the complex plane, define
\begin{equation}\label{H2}
\HH_4(z):=(z-a_1)(z-a_2)(z-a_3)(z-a_4),
\end{equation}
and suppose that $\T_n(z)=c_nz^n+\ldots\in\PP_n$ satisfies a polynomial equation of the form
\begin{equation}\label{TU2}
\T_n^2(z)-1=\HH_4(z)\,\U_{n-2}^2(z)
\end{equation}
with $\U_{n-2}(z)=c_nz^{n-2}+\ldots\in\PP_{n-2}$. Then, by \eqref{TU2}, there exists a $z^*\in\C$ such that the derivative of $\T_n$ is given by
\begin{equation}\label{z*}
\T_n'(z)=n(z-z^*)\,\U_{n-2}(z).
\end{equation}
By Theorem\,\ref{Thm-EllArcs}, $\T_n^{-1}([-1,1])$ consists of two Jordan arcs. Moreover, it is proved in \cite[Theorem\,3]{PehSch-2004} that the two Jordan arcs are crossing each other if and only if $z^*\in\T_n^{-1}([-1,1])$ (compare also Theorem\,\ref{Theorem-Connectedness}). In this case, $z^*$ is the only crossing point. Interestingly, the minimum number of analytic Jordan arcs is not always two, as the next theorem says. In order to prove this result, we need the following lemma\,\cite[Lemma\,1]{PehSch-2004}.


\begin{lemma}\label{Lemma-PehSch}
Suppose that $\T_n\in\PP_n$ satisfies a polynomial equation of the form \eqref{TU2}, where $\HH_4$ is given by \eqref{H2}, and let $z^*$ be given by \eqref{z*}.
\begin{enumerate}
\item If $z^*$ is a zero of $\U_{n-2}$ then it is either a double zero of $\U_{n-2}$ or a zero of $\HH$.
\item If $z^*$ is a zero of $\HH$ then $z^*$ is a simple zero of $\U_{n-2}$.
\item The point $z^*$ is the only possible common zero of $\HH$ and $\U_{n-2}$.
\item If $\U_{n-2}$ has a zero $y^*$ of order greater than one then $y^*=z^*$ and $z^*$ is a double zero of $\U_{n-2}$.
\end{enumerate}
\end{lemma}


\begin{theorem}
Suppose that $\T_n\in\PP_n$ satisfies a polynomial equation of the form \eqref{TU2}, where $\HH_4$ is given by \eqref{H2}, and let $z^*$ be given by \eqref{z*}. If $z^*\notin\{a_1,a_2,a_3,a_4\}$ then $\T_n^{-1}([-1,1])$ consists of two analytic Jordan arcs. If $z^*\in\{a_1,a_2,a_3,a_4\}$ then $\T_n^{-1}([-1,1])$ consists of three analytic Jordan arcs, all with one endpoint at $z^*$, and an angle of $2\pi/3$ between two arcs at $z^*$.
\end{theorem}
\begin{proof}
We distinguish two cases:
\begin{enumerate}
\item[1.] $\T_n(z^*)\notin\{-1,1\}$: By Lemma\,\ref{Lemma-PehSch}, $\T_n^2-1$ has 4 simple zeros $\{a_1,a_2,a_3,a_4\}$ and $n-2$ double zeros. Thus, by Theorem\,\ref{Thm-NuArcs}, $\T_n^{-1}([-1,1])$ consists of two analytic Jordan arcs.
\item[2.] $\T_n(z^*)\in\{-1,1\}$:
\begin{enumerate}
\item[2.1] If $z^*\in\{a_1,a_2,a_3,a_4\}$ then, by Lemma\,\ref{Lemma-PehSch}, $\T_n^2-1$ has 3 simple zeros given by $\{a_1,a_2,a_3,a_4\}\setminus\{z^*\}$, $n-3$ double zeros and one zero of multiplicity 3 (that is $z^*$). Thus, by Theorem\,\ref{Thm-NuArcs}, $\T_n^{-1}([-1,1])$ consists of three analytic Jordan arcs.
\item[2.2] If $z^*\notin\{a_1,a_2,a_3,a_4\}$ then, by Lemma\,\ref{Lemma-PehSch}, $z^*$ is a double zero of $\U_{n-2}$. Thus $\T_n^2-1$ has 4 simple zeros $\{a_1,a_2,a_3,a_4\}$, $n-4$ double zeros and one zero of multiplicity 4 (that is $z^*$). Thus, by Theorem\,\ref{Thm-NuArcs}, $\T_n^{-1}([-1,1])$ consists of two analytic Jordan arcs.
\end{enumerate}
\end{enumerate}
The very last statement of the theorem follows immediately by Lemma\,\ref{Lem-1}\,(iii).
\end{proof}


Let us mention that in \cite{PehSch-2004}, see also \cite{Sch-2007a} and \cite{Sch-2009}, necessary and sufficient conditions for four points $a_1,a_2,a_3,a_4\in\C$ are given with the help of Jacobian elliptic functions such that there exists a polynomial of degree $n$ whose inverse image consists of two Jordan arcs with the four points as endpoints. Concluding this section, let us give two simple examples of inverse polynomial images.


\begin{example}\hfill{}
\begin{enumerate}
\item Let $a_1=-1$, $a_2=-a$, $a_3=a$ and $a_4=1$ with $0<a<1$ and
\[
\HH_4(z)=(z-a_1)(z-a_2)(z-a_3)(z-a_4)=(z^2-1)(z^2-a^2).
\]
If
\[
\T_2(z):=\frac{2z^2-a^2-1}{1-a^2},\quad\U_0(z):=\frac{2}{1-a^2},
\]
then
\[
\T_2^2(z)-\HH_4(z)\U_0^2(z)=1.
\]
Thus, by Theorem\,\ref{Thm-EllArcs}, $\T_2^{-1}([-1,1])$ consists of two Jordan arcs with endpoints $a_1$, $a_2$, $a_3$, $a_4$, more precisely $\T_2^{-1}([-1,1])=[-1,-a]\cup[a,1]$.
\item Let $a_1=\ii$, $a_2=-\ii$, $a_3=a-\ii$ and $a_4=a+\ii$ with $a>0$ and
\[
\HH_4(z)=(z-a_1)(z-a_2)(z-a_3)(z-a_4)=(z^2+1)((z-a)^2+1).
\]
If
\[
\T_2(z):=\frac{\ii}{a}\bigl(z^2-az+1\bigr),\quad\U_0(z):=\frac{\ii}{a},
\]
then
\[
\T_2^2(z)-\HH_4(z)\U_0^2(z)=1.
\]
Thus, by Theorem\,\ref{Thm-EllArcs}, $\T_2^{-1}([-1,1])$ consists of two Jordan arcs with endpoints $a_1$, $a_2$, $a_3$, $a_4$. More precisely, if $0<a<2$,
\[
\T_2^{-1}([-1,1])=\bigl\{x+\ii{y}\in\C:-\frac{(x-a/2)^2}{1-a^2/4}+\frac{y^2}{1-a^2/4}=1\bigr\},
\]
i.e., $\T_2^{-1}([-1,1])$ is an equilateral hyperbola (not crossing the real line) with center at $z_0=a/2$ and asymptotes $y=\pm(x-a/2)$.\\
If $a=2$, $\T_2^{-1}([-1,1])=[\ii,a-\ii]\cup[-\ii,a+\ii]$, i.e., the union of two complex intervals.\\
If $2<a<\infty$,
\[
\T_2^{-1}([-1,1])=\bigl\{x+\ii{y}\in\C:\frac{(x-a/2)^2}{a^2/4-1}-\frac{y^2}{a^2/4-1}=1\bigr\},
\]
i.e., $\T_2^{-1}([-1,1])$ is an equilateral hyperbola with center at $z_0=a/2$, crossing the real line at $a/2\pm\sqrt{a^2/4-1}$ and asymptotes $y=\pm(x-a/2)$.\\
In Fig.\,\ref{Fig_Rectangle}, the sets $\T_2^{-1}([-1,1])$ including the asymptotes are plotted for the three cases discussed above.
\end{enumerate}
\end{example}


\begin{figure}[ht]
\begin{center}
\includegraphics[scale=0.7]{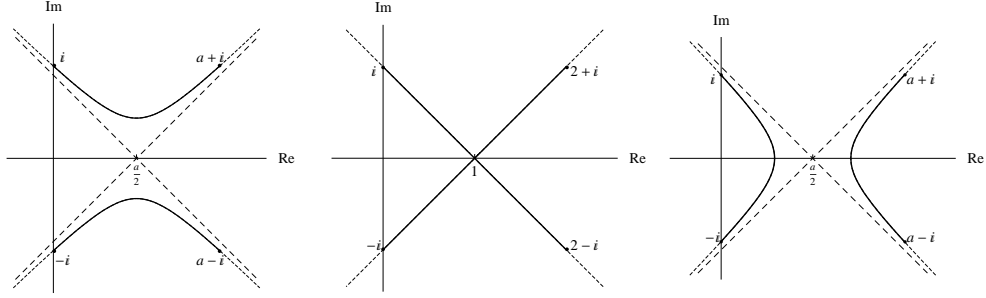}
\caption{\label{Fig_Rectangle} The inverse image $\T_2^{-1}([-1,1])$ for $0<a<2$ (left plot), for $a=2$ (middle plot) and for $a>2$ (right plot)}
\end{center}
\end{figure}

\section{The Connectedness of an Inverse Polynomial Image}


In the next theorem, we give a necessary and sufficient condition such that the inverse image is connected.


\begin{theorem}\label{Theorem-Connectedness}
Let $\T_n\in\PP_n$. The inverse image $\T_n^{-1}([-1,1])$ is connected if and only if all zeros of the derivative $\T_n'$ lie in $\T_n^{-1}([-1,1])$.
\end{theorem}
\begin{proof}
Let $\Gamma:=\bigl\{\Gamma_1,\Gamma_2,\dots,\Gamma_n\bigr\}$ denote the set of arcs of $\T_n^{-1}([-1,1])$ as in Lemma\,\ref{Lem-1}\,(vi).

``$\Leftarrow$'': Suppose that all zeros of $\T_n'$ lie in $\Gamma$. Let $A_1\in\Gamma$ be such that it contains at least one zero $z_1$ of $\T'_n$ with multiplicity $m_1\geq1$. By Lemma\,\ref{Lem-1}\,(ii), (iv) and (vi), there are $m_1$ additional arcs $A_2,A_3,\dots,A_{m_1+1}\in\Gamma$ containing $z_1$. By Lemma\,\ref{Lem-1}\,(vii),
\[
A_j\cap{A}_k=\{z_1\}~\text{for}~j,k\in\{1,2,\dots,m_1+1\},~j\neq{k}.
\]
Now assume that there is another zero $z_2$ of $\T'_n$, $z_2\neq{z}_1$, with multiplicity $m_2$, on $A_{j^*}$, $j^*\in\{1,2,\dots,m_1+1\}$. Since no arc $A_j$, $j\in\{1,2,\dots,m_1+1\}\setminus\{j^*\}$ contains $z_2$, there are $m_2$ curves $A_{m_1+1+j}\in\Gamma$, $j=1,2,\dots,m_2$, which cross each other at $z_2$ and for which, by Lemma\,\ref{Lem-1}\,(vii),
\[
\begin{aligned}
A_j\cap{A}_k&=\{z_2\} \quad &&\text{for } &&j,k\in\{m_1+2,\dots,m_1+m_2+1\}, j\neq k,\\
A_j\cap{A}_k&=\emptyset &&\text{for } &&j\in\{1,2,\dots,m_1+1\}\setminus\{j^*\},\\
& && &&k\in\{m_1+2,\dots,m_1+m_2+1\}\\
A_{j^*}\cap{A}_k&=\{z_2\} &&\text{for }
&&k\in\{m_1+2,\dots,m_1+m_2+1\}.
\end{aligned}
\]
If there is another zero $z_3$ of $\T_n'$, $z_3\notin\{z_1,z_2\}$, on $A_{j^{**}}$, $j^{**}\in\{1,2,\dots,m_1+m_2+1\}$, of multiplicity $m_3$, we proceed as before.\\
We proceed like this until we have considered all zeros of $\T'_n$ lying on the constructed set of arcs. Thus, we get a connected set of $k^*+1$ curves
\[
A^*:=A_1\cup{A}_2\cup\ldots\cup{A}_{k^*+1}
\]
with $k^*$ zeros of $\T'_n$, counted with multiplicity, on $A^*$.\\
Next, we claim that $k^*=n-1$. Assume that $k^*<n-1$, then, by assumption, there exists a curve $A_{k^*+2}\in\Gamma$, for which
\[
A_{k^*+2}\cap{A}^*=\{\}
\]
and on which there is another zero of $\T'_n$. By the same procedure as before, we get a set $A^{**}$ of $k^{**}+1$ arcs of $\Gamma$ for which $A^*\cap{A}^{**}=\{\}$ and $k^{**}$ zeros of $\T'_n$, counted with multiplicity. If $k^*+k^{**}=n-1$, then we would get a set of $k^*+k^{**}+2=n+1$ arcs, which is a contradiction to Lemma\,\ref{Lem-1}\,(i). If $k^*+k^{**}<n-1$, we proceed analogously and again, we get too many arcs, i.e., a contradiction to Lemma\,\ref{Lem-1}\,(vi). Thus, $k^*=n-1$ must hold and thus $\Gamma$ is connected.

``$\Rightarrow$'': Suppose that $\Gamma$ is connected. Thus, it is possible to reorder $\Gamma_1,\Gamma_2,\ldots,\Gamma_n$ into $\Gamma_{k_1},\Gamma_{k_2},\ldots,\Gamma_{k_n}$ such that $\Gamma_{k_1}\cup\ldots\cup\Gamma_{k_j}$ is connected for each $j\in\{2,\ldots,n\}$. Now we will count the crossing points (common points) of the arcs in the following way: If there are $m+1$ arcs $A_1,A_2,\ldots,A_{m+1}\in\Gamma$ such that $z_0\in{A}_j$, $j=1,2,\ldots,A_{m+1}$, then we will count the crossing point $z_0$ $m$-times, i.e., we say $A_1,\ldots,A_{m+1}$ has $m$ crossing points. Hence, $\Gamma_{k_1}\cup\Gamma_{k_2}$ has one crossing point, $\Gamma_{k_1}\cup\Gamma_{k_2}\cup\Gamma_{k_3}$ has two crossing points, $\Gamma_{k_1}\cup\Gamma_{k_2}\cup\Gamma_{k_3}\cup\Gamma_{k_4}$ has 3 crossing points, and so on. Summing up, we arrive at $n-1$ crossing points which are, by Lemma\,\ref{Lem-1}\,(iv) the zeros of $\T_n'$.
\end{proof}


Theorem\,\ref{Theorem-Connectedness} may be generalized to the question how many connected sets $\T_n^{-1}([-1,1])$ consists of. The proof runs along the same lines as that of Theorem\,\ref{Theorem-Connectedness}.


\begin{theorem}
Let $\T_n\in\PP_n$. The inverse image $\T_n^{-1}([-1,1])$ consists of $k$, $k\in\{1,2,\ldots,n\}$, connected components $B_1,B_2,\ldots,B_k$ with $B_1\cup{B}_2\cup\ldots\cup{B}_k=\T_n^{-1}([-1,1])$ and $B_i\cap{B}_j=\{\}$, $i\neq{j}$, if and only if $n-k$ zeros of the derivative $\T_n'$ lie in $\T_n^{-1}([-1,1])$.
\end{theorem}


\bibliographystyle{amsplain}
\bibliography{InversePolynomialImages}

\end{document}